\documentclass{amsart}
\usepackage{macros}
\standardsettings

\newcommand{\NN}{\mathbb{N}}

\newcommand{\N}{\mathrm{N}}

\begin{document}
\title{Algorithm to detect periodicity by interleaving sequences}
\author{G. Tony Jacobs}

\begin{Abstract}

We define an algorithm which begins with an sequence of sequences, and produces a single sequence, with following property: If at least one of the original sequences has a tail that is periodic, then the output sequence has a periodic tail, and conversely. Our purpose is to supplement a result by Dasaratha et al., in which a real number input results in a countable family of sequences, with the property that at least one is eventually periodic if and only if the input is a cubic irrational number. This result, in the context of that one, attempts to address Hermite's problem, which asks for some generalization of continued fractions that can detect cubic numbers via periodicity. We note in our ending remarks how this particular trick fails to give a satisfying resolution to Hermite's problem, but we present it anyway, for anyone whom it may interest.

\end{Abstract}

\maketitle

\section{Main Result}

Let $X$ be some set. We wish to define a map from $(X^\NN)^\NN$ to $X^\NN$ with the following property. Given an input which is a sequence of sequences, we want our output to be a single sequence that is eventually periodic if and only if at least one of the sequences in the input is eventually periodic. Let $a_{i,j}$ denote the $j$-th term in the $i$-th input sequence: $((a_{i,j})_{j\geq 1})_{i\geq 1}\in(X^\NN)^\NN$. We proceed to describe an algorithm for generating an output sequence $(b_n)_{n\geq 1} \in X^\NN$.

First, we define some terms.

\begin{definition}
Let $A=(a_1,\ldots,a_n)$ be an $n$-tuple. Suppose there is some $k<n$ such that the relation $a_{i+k}=a_i$ for $i=1,\ldots,n-k$, is satisfied. Suppose in addition that this $k$ is minimal, i.e., that $h<k$ implies that $a_i\neq a_{i+h}$ for some $i=1,\ldots,n-h$. Then we say that $A$ is \emph{finite-periodic with period $k$}, or simply \emph{finite-periodic}. In this case, we call $(a_1,\ldots,a_k)$ the \emph{finite-fundamental string} of $A$.
\end{definition}

We construct our output sequence by concatenating tuples, or blocks of terms, taken from the input sequences. We denote the blocks $B_l$ for $l\geq 1$, and the length of block $B_l$ is $2^l$. Blocks are taken from input sequences, in a diagonal fashion described below, checking for periodicity at each step.

Define the sequence $(i_k)_{k\geq 1}=(1,2,1,2,3,1,2,3,4,1,2,3,4,5,\ldots)$ as our sequence of indices for the purpose of diagonalization. (Note, we can give a formula for $i_k$ in terms of triangular numbers: Let $T_n=\frac{n(n+1)}{2}$ be the $n$-th triangular number. Then, for $T_n\leq k <T_{n+1}$, we have $i_k=k-T_n+1$.)

\begin{enumerate}
	\item Begin with $k=l=1$; the variable $k$ tracks our place in the indexing sequence $(i_k)$, while $l$ is the numbering for blocks. Define a Boolean variable $P$ to indicate whether we have detected finite-periodicity that may persist, and begin with $P=\textsc{false}$.
	
	\item Create block $B_l$ by removing $2^l$ terms from the beginning of the sequence $(a_{i_k,j})_j$. We say ``removing'' to emphasize that any terms used as part of block $B_l$ will not be available for any subsequent block $B_{l+h}$. Each block taken from a sequence begins with the first unused term.

	\item Check the value of $P$:
	\begin{enumerate}
		\item If $P=\textsc{false}$, check whether the block $B_l$ is finite-periodic. If so, then set $P=\textsc{true}$, let $S$ be the finite-fundamental string of $B_l$, and proceed to step $4$. If $B_l$ is not finite-periodic, then increment $k$ by $1$ and proceed to step $4$.
		\item If $P=\textsc{true}$, then we need to check whether some previously detected finite-periodicity has persisted. In this case, we have a string $S$, which is the finite-fundamental string of some concatenation of blocks $B_h + \cdots + B_{l-1}$, with $h<l$. We now check whether the concatenation $B_h + \cdots + B_l$ is finite-periodic with finite-fundamental string $S$. If it is, proceed to step $4$. If it is not, set $P=\textsc{false}$, increment $k$ by $1$, and proceed to step $4$.
	\end{enumerate}
	
	\item Increment $l$ by $1$ and return to step $2$.
\end{enumerate}

Concatenating all of the blocks $B_1 + \cdots$, we obtain our output sequence $(b_n)$.

\begin{definition}
The process described above, viewed as a map $T:(X^{\NN})^{\NN} \to X^{\NN}$ is called the \emph{sequence interleaving algorithm}.
\end{definition}

\begin{theorem}
\label{hermite}
Given a countable family of sequences as input, the output of the sequence interleaving algorithm is an eventually periodic sequence if and only if some sequence in the input family is eventually periodic.
\end{theorem}

\begin{proof}
Indeed, suppose that $(b_n)=T\left(((a_{i,j})_j)_i\right)$ is eventually periodic, with period $m$, fundamental string $S$, and a pre-period of length $r$. Choose $l_0$ minimal so that $\sum_{d=1}^{l_0-1}2^d>r$ and $2^{l_0}>m$. Now, in considering what happens when we reach block $B_{l_0}$, we must consider two cases:

Suppose we have, after finishing with block $B_{l_0-1}$, either $P=\textsc{false}$ or $P=\textsc{true}$ and finite-fundamental string $S$ (or some cyclic permutation of $S$). In either instance, after examining block $B_{l_0}$ in step $3$ of the algorithm, we will have $P=\textsc{true}$, with finite-fundamental string $S$ (or the same cyclic permutation), and $P$ will remain in this state because the periodicity is not broken in subsequent blocks. Therefore, $k$ will not increment again, and the rest of $(b_n)$ will be a tail of whichever input sequence was the source of block $B_l$. Thus, that input sequence is eventually periodic.

As a second case, suppose that the algorithm finishes with block $B_{l_0-1}$ with $P=\textsc{true}$ and finite-fundamental string $S'$ that is different from any cyclic permutation of $S$. In this case, we will finish block $B_l$ with $P=\textsc{false}$, and it will be in block $B_{l_0+1}$ that we finally get $P=\textsc{true}$ with finite-fundamental string $S$, just as in the above case. In that case, whatever input sequence is the source of terms for block $B_{l_0+1}$ is eventually periodic.

Conversely, suppose $(b_n)$ is not eventually periodic, and suppose by way of contradiction that some input sequence $(a_{i_0,j})_j$ is eventually periodic. Since the indicator $P$ can only remain in a true state while the same string continues to repeat, the lack of eventual periodicity in $(b_n)$ implies that $P$ is in a false state infinitely often. Therefore, $k$ is incremented infinitely often, and $i_k=i_0$ infinitely often, paired each time with values of $l$ that increase without bound.

Now, if $(a_{i_0,j})_j$ is eventually periodic, it has some pre-period length and some period. Eventually, $i_k$ will equal $i_0$ often enough that the corresponding blocks use up the pre-period. Eventually after that, $i_k$ will return to $i_0$ enough times that $l$ will grow until $2^l$ is greater than the period of $(a_{i_0,j})_j$. When these conditions are met, there is no way for the algorithm to leave the sequence $(a_{i_0,j})_j$, and $P$ will remain true, a contradiction.
\end{proof}

In the second part of the proof, a proof by contrapositive might seem more direct, but such an approach is complicated by the fact that more than one input sequence may be periodic, so we need the assumption that $(b_n)$ is not periodic to ensure sufficient returns to the putative periodic input.

We note that, in the case where some sequence in the input family is periodic, exactly one periodic input sequence appears, in its entirety, in the output sequence, and can be recovered by analysis of the output sequence. In the case where no input sequence is periodic, all of the input sequences appear in their entireties in the output sequence, and can be reconstructed from it.

\section{Remarks}

Hermite noted that the continued fraction algorithm has the property that the sequence of natural numbers it produces is eventually periodic if and only if the starting number is a quadratic irrational. He famously asked whether there might be some algorithm that produces a corresponding sort of sequence that displays periodicity for cubic irrational numbers.\cite{hermite} Many such ``generalized continued fractions'' or ``multi-dimensional continued fractions'' have been designed in attempts to do this. Dasaratha, et al., produce a countable family of sequences, generated from a real input by a countable family of algorithms, with the property some sequence in the family is eventually periodic if and only if the input is a cubic irrational.\cite{dasarathaetal}

Such a family of algorithms, composed with the sequence interleaving algorithm, produces one sequence from a real input, and this sequence will be eventually periodic if and only if the input is a cubic irrational. Thus, we can answer Hermite's question in the affirmative.

This answer, however, seems profoundly unsatisfying. The sequence that it offers is not a ``description'' of the input number in the way that a continued fraction is. It is not an arithmetic algorithm, but a higher-order combination of arithmetic algorithms. Indeed, some initial segment of the sequence comes from examining the ``wrong'' algorithms, which had been tuned to detect different cubic numbers, coming from different cubic fields. One would prefer to offer Hermite something more analogous to the continued fraction algorithm.

In what sense do continued fractions describe a number? The obvious answer is that they describe it in terms of Diophantine approximations, but in the case of quadratic irrationals, there is more to be said. The terms in the continued fraction expansion of such a number give us information about the structure of the ring of integers in the field generated by our input. In particular a number corresponds to some ideal in some subring of that field's integers. The pre-period inicates the steps of a ``reduction'' process, passing through a sequence of ideals, and the period itself corresponds to a set of finitely many ``reduced ideals''. The theory describing this is part of the so-called ``infrastructure of quadratic fields'', the development of which can be found in Mollin's book, \textit{Quadratics}.\cite{mollin}

To really emulate the relationship between the continued fraction algorithm and quadratic numbers, it would seem desirable to explore an infrastructure of cubic fields. One initial step in that direction is attempted in ``Periodic Sequences in Pure Cubic Fields''\cite{jacobs}.

\bibliographystyle{amsplain}

\bibliography{bibliography}

\providecommand{\bysame}{\leavevmode\hbox to3em{\hrulefill}\thinspace}
\providecommand{\MR}{\relax\ifhmode\unskip\space\fi MR }
\providecommand{\MRhref}[2]{%
  \href{http://www.ams.org/mathscinet-getitem?mr=#1}{#2}
}
\providecommand{\href}[2]{#2}
\begin{thebibliography}{1}

\bibitem{dasarathaetal}
K.~Dasaratha, L.~Flapan, T.~Garrity, C.~Lee, C.~Mihaila, N.~Neumann-Chun,
  S.~Peluse, and M.~Stoffregen, \emph{Cubic irrationals and periodicity via a
  family of multi-dimensional continued fraction algorithms}, Available on
  arXiv: http://arxiv.org/pdf/1208.4244.pdf.

\bibitem{hermite}
C.~Hermite, \emph{Letter to c.d.j. jacobi}, Journal f\"{u}r die reine und
  angewandte Mathematik \textbf{40} (1848), 286.

\bibitem{jacobs}
G.~Jacobs, \emph{Reduced {I}deals in {P}ure {C}ubic {F}ields}, Available on
  arXiv: http://arxiv.org/pdf/1905.00242.pdf.

\bibitem{mollin}
R.~Mollin, \emph{Quadratics}, CRC Press, Inc., Boca Raton, Florida, 1996.

\end{thebibliography}

\end{document}